\documentclass[12pt]{article}
\usepackage{enumitem}   
\usepackage{fancyhdr}
\usepackage[centertags]{amsmath}
\usepackage{amsfonts}
\usepackage{graphics}
\usepackage{graphicx}
\usepackage{wrapfig}
\usepackage{xcolor}
\usepackage{amssymb}
\usepackage{amsthm}
\usepackage{newlfont}
\usepackage{url}
\usepackage{epsfig}
\usepackage{ifthen}
\usepackage{ifpdf}
\usepackage{rotating}
\usepackage{caption}
\usepackage{amsmath}
 \usepackage{subfloat}
 \usepackage{subcaption}
\input amssym.def
\input amssym
\input cyracc.def
\def\Rn{\mathbb{R}^n}
\def\Rd{\mathbb{R}^d}

\def\R{\mathbb{R}}

\newtheorem{lemma}{Lemma}

\newtheorem{thm}{Theorem}

\newcommand{\be}{\begin{equation}}
\newcommand{\ee}{\end{equation}}
\newcommand{\ba}{\begin{array}}
\newcommand{\ea}{\end{array}}
\newcommand{\bee}{\begin{eqnarray*}}
\newcommand{\eee}{\end{eqnarray*}}
\newcommand{\bea}{\begin{eqnarray}}
\newcommand{\eea}{\end{eqnarray}}

\newcommand{\W}{\cal{W}}

\newcommand{\tx}{\tilde{x}}
\newcommand{\bx}{\bar{x}}
\newcommand{\bq}{\bar{q}}
\newcommand{\bv}{\bar{v}}

\theoremstyle{definition}
\newtheorem{rmk}{Remark}[section]

\newcommand{\tu}{\tilde{u}}
\newcommand{\bu}{\bar{u}}

\newcommand{\dm}{d_{min}}
\newcommand{\dM}{d_{max}}
\newcommand{\grad}[1]{\nabla F(x^{#1})}
\begin{document}

\title{Linear Convergence Rate Analysis of a Class of Exact First-Order Distributed Methods for Weight-balanced Time-Varying Networks and Uncoordinated Step Sizes}
\author{Greta Malaspina\footnote{Department of Industrial Engineering of Florence, Università degli Studi di Firenze, Italy.
Member of the INdAM Research Group GNCS. Email: greta.malaspina@unifi.it.}, Du\v{s}an Jakoveti\'c \footnote{Department of Mathematics and Informatics, Faculty of Sciences, University of Novi Sad, Trg Dositeja
Obradovi\'ca 4, 21000 Novi Sad, Serbia, Email: dusan.jakovetic@dmi.uns.ac.rs, natasak@uns.ac.rs. \phantom{aaaaaaaaaaaaaaaaaaaaaaaaaaaaaaaaaa}
This work is supported by the European Union’s Horizon 2020 programme under the Marie Sk\l{}odowska-Curie Grant Agreement no. 812912. The work of Jakoveti\'c and Kreji\'c is supported by the Provincial Secretariat for Higher Education and Scientific
Research, grant no 142-451-2593/2021-01/2.}, Nata\v{s}a Kreji\'c\footnotemark[2]}
\date{}
\maketitle

\begin{abstract}

\noindent We analyze a class of exact distributed first order methods under a
general setting on the
underlying network and step-sizes. In more detail, we allow simultaneously
for  time-varying uncoordinated step sizes and time-varying directed
weight-balanced networks,
jointly con\-nected over bounded intervals.
 The analyzed class of methods subsumes several existing algorithms
like the unified Extra and unified DIGing (Jakovetic, 2019), or the
exact spectral gradient method (Jakovetic, Krejic, Krklec Jerinkic,
2019) that have been analyzed before under more restrictive assumptions. Under the assumed setting, we establish R-linear convergence of the methods and
present several implications that our results have on the literature. Most notably, we show that the unification strategy in (Jakovetic, 2019) and the spectral step-size selection strategy in (Jakovetic, Krejic, Krklec Jerinkic, 2019) exhibit a high degree of robustness to uncoordinated time-varying step sizes and to time-varying networks. \\


\end{abstract}

\section{Introduction}
\label{sec:introduction}
We consider a set of $n$ computational agents and the following unconstrained optimization problem 
\begin{equation}\label{problem}
\min_{\mathbb{R}^d} f(y) \phantom{space} f(y) = \sum_{j=1}^n f_j(y)
\end{equation}
where each $f_j$ is a real-valued function of $\Rd$ and is held privately by one of the agents, and the agents can communicate according to a given network. Problems of this form arise in many practical applications such as sensor networks \cite{sensornetworks}, distributed control \cite{mota}, distributed learning \cite{learning} and many others.\\
Several distributed methods \cite{harnessing, extra, xu, diging, lessard, dusan, arxivVersion,cdc-submitted,DQNSIAM} have been proposed in literature for the solution of \eqref{problem} that achieve exact convergence to the minimizer with fixed step-size, when the objective function is convex and Lipschitz-differentiable. 

 In \cite{diging,harnessing} and \cite{extra} two exact gradient-based methods where proposed, and the convergence was proved for the case where the underlying network is undirected, connected, and remains constant through the entire execution of the algorithm. In \cite{dusan} a unified analysis of a class of first-order distributed methods is presented. In \cite{lessard} the convergence of several first-order methods was generalized to the case of a time-varying network, provided that the network is connected at each iteration, while in \cite{diging} the convergence analysis of \cite{harnessing} is extended to the time-varying and directed case, assuming joint-connectivity of the sequence of networks and weight-balance of each graph\footnote{That is, we assume that it is possible to define a doubly stochastic consensus matrix associated with each of the networks.}.
Interestingly, the exact first order methods are also related with augmented Lagrangian algorithms, e.g., \cite{Patrascu}. For example, the methods in \cite{diging,harnessing} and \cite{dusan} have been shown to be equivalent to certain primal-dual methods that optimize an augmented Lagrangian function associated with the original problem. In \cite{Khan3} an accelerated gradient-based method for the time-varying directed case is proposed, with weaker assumptions over the underlying networks. In \cite{mirror} the authors propose a mirror descent method that assumes time-varying jointly-strongly-connected networks. In \cite{NEXT} the authors considered the problem of minimizing $f(y)+G(y)$ over a closed and convex set $\mathcal K,$ where $f$ is a possibly non-convex function as in \eqref{problem} and $G$ is a convex non-separable term, and they propose a gradient-tracking method that achieves convergence in the case of time-varying directed jointly-connected networks for diminishing synchronized step sizes. In \cite{ScutariGen} the method proposed in \cite{NEXT} is extended with constant step-sizes to a more general framework while in \cite{ScutariRlinear} $R$-linear convergence is proved for \cite{ScutariGen} with strongly convex $f(y).$ A unifying framework of these methods is presented in \cite{ScutariFramework} and, for the case of constant and undirected networks, in \cite{Sayed2}. In all the above methods the sequence of the step-sizes is assumed to be fixed and coordinated among all the agents. In \cite{diging2}, \cite{xu}, \cite{Khan1}, \cite{Khan2}, and \cite{Sayed1} the case of uncoordinated time-constant step sizes is considered, that is, each node has a different step-size but these step sizes are constant in all iterations.  In \cite{dsg} a modification of \cite{harnessing} is proposed, with step-sizes varying both across nodes and iterations, and it is proved that there exist suitable safeguards for the steps, depending on the regularity properties of the objective function and the network, such that $R$-linear convergence of the generated sequence to the solution of \eqref{problem} holds. This result is obtained for undirected and stationary network. In \cite{ScutariAS1} and \cite{ScutariAS2} asynchronous modifications of \cite{ScutariGen} are proposed. \\

Of special interest to the current paper is the spectral gradient method (or Barzilai and Borwein method). This method is very popular in centralized optimization due to its efficiency, as reported in numerous studies, for example \cite{Dai,DiSerafino}. 
In general,  the method avoids the famous zig-zag behaviour of the steepest descent and converges much faster. 
The method was first proposed by Barzilai and Borwein \cite{BB1}. This reference proves the method's convergence for two-dimensional problems and convex quadratic functions. The analysis is then extended to arbitrary dimensions and convex quadratic functions by Raydan \cite{Raydan1}. Minimization of generic functions is considered in \cite{Raydan2} in combination with a nonmonotone line search.   
 R-linear convergence for convex quadratic functions has been proved in \cite{Dai2}. In summary, despite its excellent numerical performance, spectral gradient methods are proved to converge without any safeguarding lower and upper bounds on the step size only for strongly convex quadratic costs. Convergence for generic functions beyond convex quadratic is proved only under step size safeguarding, coupled with a line search strategy. Distributed variants of spectral gradient methods and \emph{fixed network topologies} are studied in~\cite{dsg}. 

We now summarize this paper's contributions. 
 We establish R-linear convergence of a class of exact distributed first-order methods under the general setting of 
time-varying directed weight-balanced networks, without the requirement of network connectedness at each iteration, and  in the presence of 
time-varying uncoordinated step sizes.  While there have been several existing studies of exact distributed methods under general settings, our study implies several new contributions to the literature; these contributions cannot be derived from existing works and are novelties of this paper.

\begin{itemize}
\item We prove that the methods proposed in \cite{dusan}, referred to here (and also in \cite{lessard}) as the unified Extra and the unified DIGing are robust to 
time-varying directed networks and time-varying uncoordinated step sizes, i.e., they converge R-linearly in this setting. Up to now,  it is only known that these methods converge under static undirected networks \cite{dusan} or time-varying networks where the network is connected at each iteration \cite{lessard}. These methods  have been previously considered only for time-invariant coordinated step sizes.
\item We prove that the method proposed in \cite{dsg} is robust to time-varying directed networks. Before the current paper, 
the method was only known to converge for static, undirected networks.
\item It is shown in  \cite{lessard}  that the Extra method \cite{extra} may diverge over time-varying networks, even when the network is connected at every iteration. On the other hand, as we show here, the unified Extra, a variant of Extra proposed in \cite{dusan}, is robust to time-varying networks. Hence, our results reveal that the unified Extra can be considered as a  mean to modify Extra and make it robust.
\item  We provide a thorough numerical study and an analytical study for a special problem structure that demonstrates that the unification strategy in \cite{dusan} and the spectral gradient-like step-size selection strategy in \cite{dsg} exhibit a high degree of robustness to time-varying networks and uncoordinated time-varying step-sizes. More precisely, we show that these strategies converge, when working on time-varying networks, for wider step-size ranges than commonly used strategies such as constant coordinated step-sizes and DIGing algorithmic forms. In addition, we show by simulation that actually a combination of the unification and the spectral step-size strategies further improves robustness. 
\end{itemize} 
 Technically, while considering weight-balanced digraphs instead of undirected graphs does not lead to a significant analysis difference, major technical differences here with respect to prior work correspond to the analysis of the unification strategy \cite{dusan} under time varying networks and time-and-node-varying step-sizes and spectral strategies \cite{dsg} under time-varying networks.\\

This paper is organized as follows. In Section 2 we describe the computational framework that we consider and we present the methods that we analyse. In Section 3 we recall a few preliminary results from the literature and we prove a convergence theorem for the methods introduced in Section 2. In Section 4, we show analytically and by simulation that the unification and spectral step-size selection strategies increase robustness of the methods to time-varying networks and uncoordinated step-sizes. Finally, in Section 5, we conclude the paper and outline some future research directions.

\section{The Model and the Class of Considered Methods}
We make the following regularity assumptions for the local cost functions $f_i.$

{  \noindent {\bf Assumption A1.}}
\begin{itemize}
\item Each  function $ f_{i} : \mathbb{R}^d \rightarrow \mathbb{R}$, $i=1,\dots,n, $
is twice continuously differentiable; 
\item There exists $0\leq\mu_i\leq L_i$ such that for every $ i=1,\ldots,n $ and every $ y \in \mathbb{R}^d, $ 
\be \label{A11}  \mu_i I \preceq \nabla^2 f_i (y) \preceq L_i I \ee
\end{itemize}
where  we write  $A\preceq B$ if the matrix $B-A$ is positive semi-definite.
That is, we assume that each of the local functions is $\mu_i$-strongly convex, and has Lipschitz continuous gradient with constant $L_i$. Denoting with $L= \sum_{i=1}^n  L_i $ and $\mu = \sum_{i=1}^{n} \mu_i,  $  we have that the aggregate function $f$ is $\mu$-strongly convex and $\nabla f$ is Lipschitz-continuous with the constant $L.$\\

Given $x_1,\dots,x_n\in\Rd$ we define
\begin{equation}\label{glob_vars}
x:=\left(\begin{matrix}x_1\\ \vdots\\ x_n
\end{matrix}\right)\in\R^{nd} \phantom{space} F(x) = \sum_{j=1}^n f_j(x_j).
\end{equation}

We denote with $e$ the vector of length $n$ with all components equal to 1. For a  matrix $A\in\mathbb{R}^{n\times n}$ we denote with $\nu_{\max}(A)$ the largest singular value of $A$. Moreover, given a sequence of matrices $\{M^k\}_k$ and $m\in\mathbb{N}$, let 
\begin{equation}
M^k_m:=M^kM^{k-1}\dots M^{k-m+1}, \phantom{sp} 
M^k_0 = I
\end{equation}
It is assumed that at iteration $k$ the $n$ agents are the nodes of a given network $G^k = (\{1,\dots,n\}, E_k)$, where $E_k$ denotes the set of the edges of the network, and to each $G^k$ we associate a consensus matrix $W^k\in\R^{n\times n}$.  The assumptions over the sequences $\{G^k\}$ and $\{W^k\},$ which are the same hypotheses considered in \cite{diging}, are stated below. \\

{  \noindent {\bf Assumption A2.}}\\
For every $k =0,1,\ldots$, $G^k= (\{1,\dots,n\}, E_k)$ is a directed graph and $W^k$ is an $n\times n$ doubly stochastic matrix with $w_{ij}=0$ if $i\neq j$ and $(i,j)\notin E_k$. Moreover, there exists a positive integer $m$ such that $\sup_{k=0:m-1}\nu_k<1$, where $\nu_k = \nu_{max}(W^k_m-\frac{1}{n}ee^T)$.\\

\begin{rmk}
Assumption A2 is weaker than requiring each graph  $G^k$ to be connected. For example, it can be proved (see \cite{diging}) that in the case of undirected networks, if the sequence is jointly-connected then we can ensure Assumption A2 by taking $W^k$ as, e.g., the Metropolis matrix, \cite{metropolis}, associated with $G^k.$ 
In more detail, the following can be shown. Assume that the positive entries of the weight matrices $W_k$'s are always bounded from below by a positive constant $\underline{w}$- (including also the diagonal entries, i.e., assume that the diagonal entries 
of $W_k$ are always greater than or equal to $\underline{w}$). Furthermore, assume network connectedness over bounded intercommunication intervals. 
  That is, for any fixed iteration $k$, consider the graph 
 $G_{k}^m = (\{1,\dots,n\}, E_k^m)$, $E_k^m = \cup_{\ell=k-m+1}^{k}E_k$, 
 whose set of links is the union of the sets of links of graphs at time instances 
 $\ell=k-m+1,...,k$. Assume that  $G_{k}^m$ is strongly connected, for every $k$.  
  Now, it is easy to show that the above assumptions imply that 
  $\nu_{\mathrm{max}}\left( W_k^m-\frac{1}{n}ee^T\right)<1$.\footnote{To see this, first note that, clearly, matrix $W_k^m$ is doubly stochastic. Furthermore, 
 it is easy to show (e.g., by induction) that, for any $(i,j) \in E_k^m$, 
 and for any $i=1,...,n$, 
 we have that $[W_k^m]_{ij} \geq \underline{w}^m$. 
  This means that $W_k^m$ is a doubly stochastic matrix 
  with positive diagonal entries, and, moreover, all its off-diagonal entries 
  at the positions that correspond to links of $G_k^m$ are strictly positive. Using standard 
  arguments on doubly stochastic matrices, this implies 
  that $\nu_{\mathrm{max}}\left( W_k^m-\frac{1}{n}ee^T\right)$.} 
\end{rmk}

We also comment on the role of quantity $m$ on the convergence of \eqref{it}. 
Our main result, Theorem~2 ahead, certifies that there exists choice of step size 
lower and upper bounds $d_{min}$ and $d_{max}$ such that R-linear convergence of the method \eqref{it} holds. The result holds for any choice of~$m$. 
 Clearly, the specific values of $d_{min}$ and $d_{max}$ 
 in general depend on~$m$. Intuitively, 
 we can expect that for larger $m$, the maximal admissible step-size is lower; 
 also, for fixed step-size choices $d_{min}$ and $d_{max}$ that lead to R-linear convergence, larger $m$ leads to slower R-linear convergence, 
 i.e., it leads to a worse R-linear convergence factor.
 
We consider the following class of methods. Assume that at each iteration node $i$ holds two vectors $x_i^k$ and $u_i^k$ in $R^d$ and that the global vectors $x^k,\ u^k\in\R^{nd}$, defined as in \eqref{glob_vars}, are updated according to the following rules:
\begin{equation}\label{it}
\begin{cases}
x^{k+1} = \mathcal W^kx^k-D^k(u^k+\nabla F(x^k))\\
u^{k+1} = u^k + (\mathcal W^k-I)(\nabla F(x^k)+u^k-B^kx^k)
\end{cases}
\end{equation}
where $\W^k:=(W^k\otimes I)\in\R^{nd\times nd}$, $D^k = diag(d^k_1I,\dots,d^k_nI)$ with $d^k_i$ being the step-size for node $i$ at iteration $k$ and $B^k$ is a symmetric $n\times n$ matrix that respects the sparsity structure of the communication network $G^k$ and such that for every $y\in\R^d$ we have $B^k(1\otimes y) = c(1\otimes y)$ for some $c\in\R$. Moreover, we assume that $x^0\in\R^{nd}$ is an arbitrary vector and $u^0 = 0\in\R^{nd}$.  \\
For $B^k = 0$ and appropriate choice of the step-sizes $d^k_i$ we get the method introduced in \cite{dsg}.
For $D^k = \alpha I$, if $ B^k = bI $ or $B^k = b\W$ we retrieve the class of methods analyzed in \cite{dusan} while if $B_k = 0$ we retrieve the DIGing method proposed in \cite{diging, harnessing}. For $D^k = \alpha I$ and $B^k = b\W$ with $b = \frac{1}{\alpha}$ we have the EXTRA method \cite{extra}, but while this method can be described with this choice of the parameters in equation \eqref{it}, it is not included in the class of methods we consider. Namely, the theoretical analysis that we carry out in Section 3 requires the parameter $b$ to be independent on the step-sizes, thus ruling out the choice $b=\frac{1}{\alpha}$ that yields EXTRA method. This is in line with \cite{lessard} that shows that EXTRA may not converge in general for time-varying networks. \\
 In our analysis, we consider the case $B^k = bI$ and $B^k = b\mathcal W^k$ with $b$ non-negative constant and $d_{min}\leq d^k_j\leq d_{max}$ for every $k$ and every $j=1,\dots,n$ for appropriately chosen  safeguards $0<\dm<\dM.$ \\

A possible choice for uncoordinated and time-varying step-sizes was proposed in \cite{dsg} where we have $d^k_i = (\sigma_i^k)^{-1}$  with $\sigma_i^k$ given by
\begin{equation*}
\begin{aligned}
\sigma^{k}_i =& \mathcal{P}_{[\sigma_{min}, \sigma_{max}]}\Bigg(\frac{(s_i^{k-1})^{T}y_i^{k-1}}{(s_i^{k-1})^{T}s_i^{k-1}}
+\sigma^{k-1}_i\sum_{j=1}^nw^k_{ij}\left(1-\frac{(s_i^{k-1})^{T}s_j^{k-1}}{(s_i^{k-1})^{T}s_i^{k-1}}\right)\Bigg)
\end{aligned}
\end{equation*}
where $s^{k-1}_i = x^{k}_i-x^{k-1}_i$ and $y^{k-1}_i = \nabla f_i(x^k_i) - \nabla f_i(x^{k-1}_i). $ 
Here, $\mathcal{P}_U$ denotes the projection onto the closed set U, 
$\sigma_{min}=1/\dM$, and $\sigma_{max} =1/d_{min}$. We refer to \cite{dsg} for details on the derivation and intuition behind this step-size choice. 

For static networks, this step size choice incurs no communication overhead 
per iteration; see \cite{dsg} However, for time-varying networks, 
the communication and storage protocol to implement this step size needs to be adapted.
One way to ensure at node $i$ and iteration $k$ the availability of $s_j^{k-1}$ for $(i,j)\in E^k$,
is that node $i$ receives $s_j^{k-1}$ for all $j$ such that $(i,j) \in E^k$. 
That is, each node $j$ per iteration additionally broadcasts one $d$-dimensional vector $s_j^k$ 
to all its current neighbors. 
Therefore the method described by equation \eqref{it} combines \cite{dusan} and \cite{dsg} into a more general method.

\section{Convergence Analysis}
We now study the convergence of the method described in \eqref{it}. Specifically, denoting with $y^*$ the solution of \eqref{problem} and defining
\begin{equation*}
x^*:=\left(\begin{matrix}y^*\\ \vdots\\ y^*
\end{matrix}\right)\in\R^{nd}
\end{equation*}
we prove that, if Assumptions A1 and A2 hold, there exist  $0<\dm<\dM$ such that the sequence $\{x^k\}$ generated by \eqref{it} converges to $x^*$.\\

Given a vector $v\in\R^{nd},$ denote with $\bv$ the average $\bv=\frac{1}{n}\sum_{j=1}^{n}v_j\in\R^{d}$ and with $J$ the $n\times n$ matrix $(I-\frac{1}{n}ee^T)$, where $e^T = (1,\dots,1)\in\Rn$.
Recalling the definition of $x^k$ and $u^k$ given in \eqref{it}, we define the following quantities, which will be used further on: \\
\begin{align}
&\tx^k = x^k-e\bx^k\in\R^{nd},\\
&\tu^k = u^k+\nabla F(x^*)\in\R^{nd},\\
&q^k = x^k-x^* = \tx^k+e\bq^k\in\R^{nd}.
\end{align}

To simplify the notation, in the rest of the section we assume that the dimension $d$ of the problem is given by $d=1$, but the same results can be proved analogously in the general case.\\
A few results listed below will be needed for the convergence result presented in this paper. 
Since $W^k$ is doubly stochastic, we have that $\frac{1}{n}ee^t(W^k-I) = 0$. Using this equality and the definition of $u^{k+1}$ we get
\begin{align*}
&\bu^{k+1} = \frac{1}{n}ee^Tu^{k+1} = \frac{1}{n}ee^Tu^k + \frac{1}{n}ee^T(W^k-I)(u^k+\nabla F(x^k) - B^kx^k) = \bu^k
\end{align*}
and by the initialization $u^0=0$, we have that
\begin{equation}\label{avgu1}
\bu^k = 0.
\end{equation}
Directly by the definition of $\tu^k$ and \eqref{avgu1} we get
\begin{align}
\label{avgu2}&\frac{1}{n}ee^T\tu^{k} = \frac{1}{n}ee^T(u^{k} + \nabla F(x^*)) = 
\bu^k + \frac1n\nabla f(y^*)  = 0.
\end{align}
From Assumption A1, for every $k$ there exists a matrix $H_k\preceq LI$ such that
\begin{equation}\label{Hk}
\nabla F(x^k)-\nabla F(x^*)=H_k(x^k-x^*).
\end{equation}

\begin{lemma} \label{lemmaW} \cite{diging} 
If the matrix sequence $\{W^k\}_k$ satisfies assumption A2, then for every $k\geq m$ we have
$$\left\|JW^k_my\right\|\leq\nu_k\left\|Jy\right\|$$
\end{lemma}

\begin{lemma}\label{lemmaGM}\cite{GD}
If the function $f$ satisfies assumption A1 and $0<\alpha<\frac{1}{L}$, then
$$\|y-\alpha\nabla f(y)-y^*\|\leq\tau\|y-y^*\|$$
where $\tau = \max\{|1-\alpha\mu|,|1-\alpha L|\}$
\end{lemma}

Following the idea presented in  \cite{diging}, our convergence result  relies on the Small Gain Theorem \cite{smallgain}, which we now briefly recall.
 Denote by $ \mathbf{a}:=\{a^k\}$ an infinite sequence of vectors, $ a^k \in \mathbb{R}^d$ for $ k=0,1,\ldots. $ For a fixed $ \lambda \in (0,1) $ we define
$$
\|\mathbf{a}\|^{\lambda, K} = \max_{k=0,1,\ldots, K} \left\{\frac{1}{\lambda_k}\|a^k\|\right\}
$$
$$
\|\mathbf{a}\|^{\lambda} = \sup_{k \geq 0} \left\{\frac{1}{\lambda_k}\|a^k\|\right\}.
$$

\begin{thm} \label{tesmallgain}\cite{smallgain}.
Let $ \mathbf{a} = \{a^k\}$ and $\mathbf{b} = \{b^k\}$ be two vector sequences, with $ a^k, b^k \in \mathbb{R}^d.$ If there exists $ \lambda \in (0,1)$ such that for all $ K = 0,1,\ldots, $ the following inequalities hold
\begin{equation}\label{smallgainHP}
\begin{aligned}
&\|\mathbf{a}\|^{\lambda,K} \leq \gamma_1 \|\mathbf{b}\|^{\lambda, K} + w_1,\\
&\|\mathbf{b}\|^{\lambda,K} \leq \gamma_2 \|\mathbf{a}\|^{\lambda, K} + w_2,
\end{aligned}
\end{equation}
with $ \gamma_1\cdot\gamma_2 \in [0,1) $, then
$$
\|\mathbf{a}\|^{\lambda} \leq \frac{1}{1-\gamma_1 \gamma_2}(w_1 \gamma_2 + w_2).
$$
 and 
$$\lim_{k \to \infty} a^k =0\phantom{spa}  \text{R-linearly}.$$
\end{thm}

We will use the following technical Lemma to show that the sequences $\|\bq^k\|$ and $\|\tx^k\|$ satisfy the hypotheses of Theorem \ref{tesmallgain}.

\begin{lemma}\label{conditions}
Given $b, \mu, L \geq 0$, $\nu\in(0,1)$ and $n,m\in\mathbb N$, where we denote with $\mathbb N$ the set of positive integers, there exists $\lambda\in(0,1)$ and $0\leq\dm<\dM$ such that the following conditions hold:
\begin{enumerate}
\item \label{1}$\nu<\lambda^m;$
\item \label{2}$ \frac{\dm}{n}<\frac{2}{L};$
\item \label{3}$1-\mu\dm+\Delta L<\lambda;$
\item \label{4}$\gamma\beta_2<1;$
\item \label{5}$\beta_3<1;$
\item  \label{6}$\frac{\beta_5\gamma}{1-\beta_3}<1;$
\item  \label{7}$\frac{\beta_1+\gamma\beta_2}{1-\gamma\beta_2}\cdot\frac{\beta_4+\gamma\beta_5}{1-\beta_3-\gamma\beta_5}<1,$

\end{enumerate}
where
\begin{align*}
&\gamma = \frac{(b+L)C}{\lambda^m-\nu},
 & & \beta_1 = \frac{L\dM}{\lambda-1+\mu\dm-\Delta L},\\ 
 &\beta_2 = \frac{\Delta}{L\dM}\beta_1, & & 
\beta_3 = \frac{\nu}{\lambda^m}+\beta_4, \\ 
&\beta_4 = L\beta_5, &  &\beta_5 = \frac{C\dM}{\lambda^m},\\
&\Delta = \dM-\dm, & & C = \frac{\lambda(1-\lambda^m)}{1-\lambda}.
\end{align*}
\end{lemma}
\begin{proof}
Take $\lambda^m>\nu$ and $\dm<\frac{2n}{L}$ so that \ref{1}. and \ref{2}. hold. For  $\dM>\dm$ and  close enough to $\dm$ one can ensure that
\begin{equation}\label{dmaxdmin}
\frac{\dM}{\dm}<1+\frac{\mu}{L}
\end{equation}
holds. By the previous inequality, we have $1-\mu\dm+\Delta L<1$ and therefore, for fixed $\dM$ and $\dm$ we can always take $\lambda\in(0,1)$ such that \ref{3}. is satisfied and \ref{1}. still holds. Moreover, we can take $\dm$ arbitrarily small and $\dM$ arbitrarily close to $\dm$ without violating conditions \ref{1}.-\ref{3}. Notice that
$C = \frac{\lambda(1-\lambda^m)}{1-\lambda}$ is an increasing function of $\lambda$.\\
Let us now consider condition \ref{4}. given by
$$ \frac{(b+L)C\Delta}{(\lambda^m-\nu)(\lambda-1+\mu\dm-\Delta L)}<1.$$
The left hand side expression  is an increasing function of $\Delta$ and it is equal to 0 for $\Delta=0. $ Therefore, taking $\dM$ close enough to $\dm$, condition \ref{4}. holds.\\
Condition \ref{5}. holds for $\dM<\frac{\lambda^m-\nu}{\lambda^mLC}$.\\
Consider now condition \ref{6}., 
$$ \frac{(b+L)C^2\dM}{(\lambda^m-\nu)(\lambda^m-\nu)(\lambda^m-\nu-L\dM C)}<1$$
The left hand side  expression  is an increasing function of $\dM$ and taking $\dM$ small enough we conclude that the previous inequality holds. Since we need $\dM>\dm$, in order to be able to take $\dM$ small, we need to take $\dm$ small enough, but this can be done without violating the previous conditions.\\

By definition, $\frac{\beta_2+\gamma\beta_3}{1-\gamma\beta_3}$ and $\frac{\beta_5+\gamma\beta_6}{1-\beta_4-\gamma\beta_6}$ are also increasing functions of $\dM$ and $\Delta$. Thus, we can apply the same reasoning that we applied to \ref{4}. and \ref{6}. to get $\gamma_2<1$ and $\gamma_3<1$. In particular, we can take $\dm$ and $\dM$ such that condition \ref{7}. holds.
$\ $\\
\end{proof}

\begin{thm}\label{thm}
Let $ B^k  $ be defined as $ B^k = b W^k$ or  $B^k = bI $ for a positive constant $ b, $ or $ B^k = 0. $ If Assumptions A1 and A2 hold then there exists  $\dm<\dM$ such that the sequence $\{x^k\}$ generated by \eqref{it} converges $R$-linearly to $x^*.$
\end{thm}
\begin{proof}
Define $\nu = \displaystyle\sup_{k=0:m-1}\nu_k<1$ where $\nu_k,\ m$ are given in assumption A2, and take $\lambda\in(0,1), 0\leq\dm<\dM$ given by Lemma \ref{conditions}. We prove that $n^{1/2}\bq^k$ and $\tx^k$ satisfy inequalities \eqref{smallgainHP}, thus ensuring $R$-linear convergence by Theorem \ref{tesmallgain}.\\

We have $B^k=bI$ or $B^k = bW^k$, in both cases, $B^kx^* = bx^*$, therefore 
$(W^k-I)B^kx^* = 0$ and thus
$$(W^k-I)B^kx^k = (W^k-I)B^k(x^k-x^*) = (W^k-I)B^kq^k$$
For $k\geq m-1$, using \eqref{it}, the previous equality and \eqref{Hk}, we get
\begin{equation}\label{utilde}\begin{aligned}
&\tu^{k+1}=u^{k+1}+\grad{*} =\\
&= u^k + (W^k-I)(u^k+\grad{k}-B^kx^k)+\grad{*}=\\
&=W^k(u^k+\grad{*})+(W^k-I)(\grad{k}-\grad{*})+\\
&\ -(W^k-I)B^kx^k=\\
&=W^k\tu^{k}+(W^k-I)H_kq^k-(W^k-I)B^kq^k = \\
&=W^k_m\tu^{k-m+1}+\sum_{t=0}^{m-1}W^k_t(W^{k-t}-I)\left(H_{k-t}-B^{k-t}\right)q^{k-t}
\end{aligned}\end{equation}

By  \eqref{avgu2} and  Lemma \ref{lemmaW},
\begin{equation*}\begin{aligned}
\|W^k_m\tu^{k-m+1}\| &= \|W^k_mJ\tu^{k-m+1}\|\leq \nu\|J\tu^{k-m+1}\| =\\
& =\nu\|\tu^{k-m+1}\|
\end{aligned}\end{equation*}
and by \eqref{Hk}, the definition of $B^k$ and the fact that $W^k$ is doubly stochastic, we get
$$\|W^k_t(W^{k-t}-I)(H_{k-t}-B^k)q^{k-t}\|\leq (L+b)\|q^{k-t}\|.$$
Taking the norm in \eqref{utilde} and using the two previous inequalities, we have that for $k\geq m-1$
\begin{equation*}
\|\tu^{k+1}\|\leq\nu\|\tu^{k-m+1}\|+(b+L)\sum_{t=0}^{m-1}\|q^{k-t}\|.
\end{equation*}
Notice that the above inequality also holds for the third case considered, i.e. for $ B^k = 0, $ taking $ b=0 $.
Multiplying by $\frac{1}{\lambda^{k+1}}$, taking the maximum for $k=-1:\bar k-1$, and defining 
$$\tilde\omega_1 = \max_{k=-1:m-1}\left\{\frac{1}{\lambda^{k+1}}\|\tu^{k+1}\|\right\}$$
we get
\begin{equation*}\begin{aligned}
\|\tu\|^{\lambda\bar k}&= \max_{k=-1:m-1}\left\{\frac{1}{\lambda^{k+1}}\|\tu^{k+1}\|\right\}+\max_{k=m:\bar k}\left\{\frac{1}{\lambda^{k+1}}\|\tu^{k+1}\|\right\} \leq\\
&\leq \frac{\nu}{\lambda^m}\max_{k=m:\bar k}\left\{\frac{1}{\lambda^{k-m+1}}\|\tu^{k-m+1}\|\right\} +\\
&\ + (b+L)\sum_{t=0}^{m-1}\frac{1}{\lambda^t}\max_{k=m:\bar k}\left\{\frac{1}{\lambda^{k-t}}\|q^{k-t}\|\right\}+ \tilde\omega_1\leq\\
&\leq \frac{\nu}{\lambda^m}\|\tu\|^{\lambda\bar k} +\frac{(b+L)}{\lambda^m}\frac{\lambda(1-\lambda^m)}{(1-\lambda)}\|q\|^{\lambda\bar k} + \tilde\omega_1.
\end{aligned}\end{equation*}

Since by condition 1. in Lemma \ref{conditions} we have $\nu<\lambda^m$, reordering the terms in the previous inequality and using the fact that $q^k = \tx^k+e\bq^k$, we get
\begin{equation}\label{tugamma1}\begin{aligned}
\|\tu\|^{\lambda\bar k}&\leq\gamma_1\|q\|^{\lambda\bar k} + \omega_1 \leq\\
&\leq\gamma_1\|\tx\|^{\lambda\bar k}+\gamma_1n^{1/2}\|\bq\|^{\lambda\bar k} + \omega_1 
\end{aligned}\end{equation}
with 
$$\gamma_1 = \frac{(b+L)\lambda(1-\lambda^m)}{(1-\lambda)(\lambda^m-\nu)}, \phantom{space} \omega_1 = \frac{\lambda^m}{\lambda^m-\nu}\tilde\omega_1. $$

Let us now consider $\bq^k$.
\begin{equation*}\begin{aligned}
\bq^{k+1} &= \bx^{k+1}-y^* = \frac{1}{n}ee^Tx^{k+1}-y^*=\\
&=\frac{1}{n}ee^T\left(W^kx^k-D^k(u^k+\nabla F(x^k)\right)-y^* =\\
&= \bx^k-y^* -\frac{\dm}{n}\nabla F(\bx^k)+\\ &+\frac{\dm}{n}\sum_{j=1}^n(\nabla f_j(\bx^k)-\nabla f_j(y^*))+\\&-\frac{1}{n}\sum_{j=1}^n(d_j^k-\dm)(\nabla f_j(x^k_j)-\nabla f_j(y^*)) +\\ &+ \frac{1}{n}\sum_{j=1}^n(\dm-d^k_j)\tu^k_j.
\end{aligned}\end{equation*}
Taking the norm, by Lipschitz continuity of the gradient and denoting with $\Delta = \dM-\dm$, we have
\begin{equation*}\begin{aligned}
&\|\bq^{k+1}\|= \left\|\bx^k-y^* -\frac{\dm}{n}\nabla F(\bx^k)\right\|+ \frac{L\dm}{n}\|\bx^k-y^*\|_1+\\ &+
\frac{L\Delta}{n}\|x^k-x^*\|_1 + \frac{\Delta}{n}\|\tu^k_j\|_1.
\end{aligned}\end{equation*}
Now $\frac{\dm}{n}<\frac{2}{L}$, thus Lemma \ref{lemmaGM} gives a bound for the first term in the right hand side of the last inequality, and we get
\begin{equation*}\begin{aligned}
n^{1/2}\|\bq^{k+1}\|&\leq n^{1/2}\tau\|\bx^k-y^*\|+ L\dm\|\bx^k-y^*\|+\\ &+
L\Delta\|x^k-x^*\| + \Delta\|\tu^k_j\|\leq\\
&\leq n^{1/2}(\tau+\Delta L)\|\bq^k\|+L\dM\|\tx^k\|+\Delta\|\tu^k\|.
\end{aligned}\end{equation*}
Multiplying by $\frac{1}{\lambda^{k+1}}$ and taking the maximum for $k=-1:\bar k-1$ we get
\begin{equation*}
n^{1/2}\|\bq\|^{\lambda\bar k}\leq \frac{\tau+\Delta L}{\lambda}n^{1/2}\|\bq\|^{\lambda\bar k}+
\frac{L\dM}{\lambda}\|\tx\|^{\lambda\bar k} +
\frac{\Delta}{\lambda}\|\tu\|^{\lambda\bar k}.
\end{equation*}

By Lemma \ref{conditions} we have $\tau = 1-\mu\dm$ and $\tau+\Delta L<\lambda$, thus reordering and using \eqref{tugamma1}, we get
\begin{equation*}
\begin{aligned}
n^{1/2}\|\bq\|^{\lambda\bar k}&\leq
\frac{L\dM}{\lambda-\tau-\Delta L}\|\tx\|^{\lambda\bar k} +
\frac{\Delta}{\lambda-\tau-\Delta L}\|\tu\|^{\lambda\bar k}\leq\\
&\leq (\beta_1+\gamma_1\beta_2)\|\tx\|^{\lambda\bar k}+\gamma_1\beta_2n^{1/2}\|\bq\|^{\lambda\bar k}+\beta_2\omega_1
\end{aligned}\end{equation*}
where $\beta_1$ and $\beta_2$ are defined  in Lemma \ref{conditions}.
Take 
$$\gamma_2 = \frac{\beta_1+\gamma_1\beta_2}{1-\gamma_1\beta_2}, \phantom{sp}  \omega_2 = \frac{\beta_2\omega_1}{1-\gamma_1\beta_2}.$$
From 4. in Lemma \ref{conditions} we get
\begin{equation}\label{bq}
n^{1/2}\|\bq\|^{\lambda\bar k}\leq\gamma_2\|\tx\|^{\lambda\bar k}+\omega_2.
\end{equation}
$\ $\\
Finally, let us  consider $\tx^k$. For $k\geq m-1$, by definition of $x^k, \tilde u^k$ and $q^k,$ and equation \eqref{Hk} we have
\begin{equation*}\begin{aligned}
&\tx^{k+1}=J(W^kx^k-D^k(u^k+\nabla F(x^k)) = \\
&=JW^kW^{k-1}x^{k-1}-JW^kD^{k-1}(u^{k-1}+\grad{k-1})+\\
&-J D^{k}(u^{k}+\grad{k})=\\
&=JW^k_mx^{k-m+1}-J\sum_{t=0}^{m-1}W^k_tD^{k-t}(\tu^{k-t}+H_{k-t}q^{k-t}).
\end{aligned}\end{equation*}
Taking the norm, applying Lemma \ref{lemmaW} and \eqref{Hk}, we get
\begin{equation*}\begin{aligned}
\|\tx^{k+1}\|&\leq\nu\|\tx^{k-m+1}\|+\dM\sum_{t=0}^{m-1}(\|\tu^{k-t}\|+L\|q^{k-t}\|).
\end{aligned}\end{equation*}

Multiplying by $\frac{1}{\lambda^{k+1}}$ and taking the maximum for $k=-1:\bar k-1$
we get
\begin{equation*}\begin{aligned}
\|\tx\|^{\lambda\bar k}&\leq\frac{\nu}{\lambda^m}\|\tx^{k-m+1}\|+\dM\frac{\lambda(1-\lambda^m)}{\lambda^m(1-\lambda)}\|\tu\|^{\lambda\bar k}+\\
&+L\dM\frac{\lambda(1-\lambda^m)}{\lambda^m(1-\lambda)}\|q\|^{\lambda\bar k}+\tilde\omega_3\leq\\
&\leq\beta_3\|\tx^{k-m+1}\|+\beta_4n^{1/2}\|\bq\|^{\lambda\bar k}+\beta_5\|\tu\|^{\lambda\bar k}+\tilde\omega_3.
\end{aligned}\end{equation*}
where $\tilde\omega_3 = \max_{k=-1:m-1}\left\{\frac{1}{\lambda^{k+1}}\|\tx^{k+1}\|\right\}$ and $\beta_3, \beta_4, \beta_5$ are defined  in Lemma \ref{conditions}. In particular, we have $\beta_3<1$, and can rearrange the terms of the previous inequality to  get 
\begin{equation*}\begin{aligned}
\|\tx\|^{\lambda\bar k}&\leq \frac{\beta_4}{1-\beta_3}n^{1/2}\|\bq\|^{\lambda\bar k}+\frac{\beta_5}{1-\beta_3}\|\tu\|^{\lambda\bar k}+\frac{\tilde\omega_3}{1-\beta_3}.
\end{aligned}\end{equation*}
Now, applying \eqref{tugamma1} and 6. from Lemma \ref{conditions}, we obtain
\begin{equation*}\begin{aligned}
\|\tx\|^{\lambda\bar k}&\leq\gamma_3n^{1/2}\|\bq\|^{\lambda\bar k}+\omega_3
\end{aligned}\end{equation*}
with
$$\gamma_3 = \frac{\beta_4+\beta_5\gamma_1}{1-\beta_3-\gamma_1\beta_5}, \phantom{sp} 
\omega_3= \frac{\tilde\omega_3+\beta_5\omega_1}{1-\beta_3-\gamma_1\beta_5}.$$
$\ $\\

We thus proved
\begin{equation*}\begin{aligned}
&n^{1/2}\|\bq\|^{\lambda\bar k}\leq\gamma_2\|\tx\|^{\lambda\bar k}+\omega_2\\
&\|\tx\|^{\lambda\bar k}\leq\gamma_3n^{1/2}\|\bq\|^{\lambda\bar k}+\omega_3\\
\end{aligned}\end{equation*}
with $\gamma_2\gamma_3<1$ by condition 7. in Lemma \ref{conditions}. By the Small Gain Theorem, we have that $\|\bq^k\|$ and $\|\tx^k\|$ converge to 0, and thus $\|q^k\|$ converges to zeros, which gives the thesis.

\end{proof}

The above theorem states the R-linear convergence of the method and hence one might naturally ask what is the convergence factor and how it compares with similar methods, in particular with DIGing.  Given that the setting here is rather general -- allowing for node-specific and iteration-specific step sizes on time varying networks,  the generality of the encompassed methods and settings makes analytical close-form expression of the convergence factor unfeasible. One comparison between the convergence factors of DIGing and a method of the class considered here under restrictive assumptions of a static network is given in \cite{dusan}, Remark 5. Therein, it is shown that the class of methods considered here can have a better convergence factor than DIGing. Specifically, the favorable convergence factor is achieved in \cite{dusan} for a method instance within the class when parameter $b$ is set differently than the choice that recovers DIGing. Although such comparison is derived for a narrow class of problems, contrary to the more general setting of DIGIng and the setting considered here, with fixed stepsizes and static networks,  it serves as an indication for the comparison between DIGing and the method considered here. Furthermore, numerical experiments presented in Section 4 contain the comparison between the method considered here and DIGing and show faster convergence and an increased degree of robustness of the method considered here.

\section{Analytical and Numerical Studies of Robustness of the Methods}
Theorem \ref{thm} and Lemma \ref{conditions} ensure convergence of the considered class of methods. Namely, they establish existence of bounds $\dm<\dM$ such that the methods converge $R$-linearly under the given assumptions. However they do not provide any information about the difference $\Delta = \dM-\dm$ and thus about how much the steps employed by different nodes and at different iterations can differ. In this section we try to address this issue by investigating in practice the length of the interval of admissible step-sizes. Firt we show a particular example where the method converges without any upper bound $\dM,$
then we present a set of numerical results that show how the step bounds influence the convergence and the performance of the methods.\\

We consider the same framework considered in \cite{dsg} (section 4.2) and we prove that even if we allow the consensus matrix to change from iteration to iteration, the method converges. Consider the following objective function
\begin{equation}\label{simpleproblem}
f(y) = \sum_{i=1}^nf_i(y)\text{ with } f_i(y) = \frac12(y-a_i)^2\text{ and } a\in\Rn \end{equation}

and assume that at iteration $k$ the consensus matrix is given by $$W_k = (1-\theta_k)I+\theta_k J, \text{ with } \theta\in(0,1).$$

\begin{lemma}
Assume that $\theta_k\in(\frac13, \frac34)$ for every $k$, and that $\{x^k\}$ is the sequence generated by \eqref{it} with $b=0$, $e^Tx_0 = e^Ta$ and $e^T(u_0+\nabla F(x_0)) = 0$.\\
If $d^k_i = \alpha$ for every $i=1,\dots,n$ and for every $k$, then the method converges $R$-linearly to the solution of \eqref{simpleproblem} if $\alpha_{min}\leq\alpha\leq\frac23$ and $\alpha_{min}>0$ small enough. On the other hand, for any $\alpha>2$, there exists a sequence $\{\theta_k\},\ k=0,1,2,\dots$ that satisfies the assumptions of the Lemma such that the method diverges, i.e., $\|x^k\| \rightarrow \infty$.\\
If $d_i^k = (\sigma^k_i)^{-1}$ with
\begin{equation}\label{spectralstep}
\sigma^{k+1}_i = \mathcal{P}_{[\sigma_{min}, \sigma_{max}]}\left(1+\sigma^k_i\sum_{j=1}^nw^k_{ij}\left(1-\frac{s^k_j}{s^k_i}\right)\right), 
\end{equation}
with $s^k_i = x^{k+1}_i-x^k_i$, $\sigma_{min} = 0$, $\sigma_{max} = 3/2$ and $\sigma^0_i = \sigma\in(\sigma_{min}, \sigma_{max})$ for every $1,\dots,n$, then $\{x^k\}$ converges $R$-linearly to the solution of \eqref{simpleproblem}.\\

\begin{proof}
In the case we are considering, \eqref{it} is equivalent to 
\begin{equation}
\begin{cases}
x^{k+1} = W^kx^k-D^kz^k\\
z^{k+1} = W^kz^k + x^{k+1}-x^k
\end{cases}
\end{equation}
where $D^k = diag(d^k_1I,\dots,d^k_nI)$.\\
Let us consider the case with fixed step-size $d^k_i = \alpha$ and let us denote with $\xi^k$ the vector $(q^k, z^k)\in\R^{2n}$. We can see that for every $k$ we have $\xi^{k+1} = A_k\xi^k$ where the matrix $A_k$ is given by
$$A_k = \left(\begin{matrix} W_k-J & -\alpha I\\ W_k-I & W_k-\alpha I
\end{matrix}\right)\in \R^{2n\times 2n}.$$ In order to prove the first part of the Lemma, it is enough to show that there exists $\mu<1$ such that $\|A_k\|^2_2<\mu$ for every iteration index $k$. That is, we have to prove that there exists $\mu<1$ such that the spectral radius of $A_k^TA_k$ is smaller than $\mu$ for every $k$. Denoting with ${1, \lambda^k_2,\dots,\lambda^k_n}$ the eigenvalues of $W^k$, it can be proved that the eigenvalues of $A_k^TA_k$ are given by the eigenvalues of the $2\times 2$ matrices $M^k_i$ defined as
$$ M^k_1 = \left(\begin{matrix} \alpha^2 & \alpha(\alpha-1) \\ \alpha(\alpha-1) & (\alpha-1)^2 
\end{matrix}\right)$$
$$
 M^k_i = \left(\begin{matrix} (\lambda^k_i)^2+\alpha^2 & (\lambda^k_i)^2-(1+\alpha)\lambda^k_i+\alpha^2 \\ (\lambda^k_i)^2-(1+\alpha)\lambda^k_i+\alpha^2& 2(\lambda^k_i)^2-2(1+\alpha)\lambda_i^k + 1+\alpha^2
\end{matrix}\right)$$
for $i=2,\dots,n.$
By direct computation we can see that the eigenvalues of $M^k_1$ are given by $0$ and $2\alpha^2-2\alpha+1<1-\frac{2}{3}\alpha_{min}$ and therefore it is enough to take $\mu>1-\frac{2}{3}\alpha_{min}.$ Denoting with $p^k_i(t)$ the characteristic polynomial of $D^k_i$ we can see that with the values of $\theta_{min}, \theta_{max}$ and $\alpha_{max}$ given by the assumptions, we can always find $1-\frac{2}{3}\alpha_{min}<\mu<1$ such that $p^k_i(\mu)>0$ and $p^k_i(-\mu)>0$ and thus such that the eigenvalues of $M^k_i$ belong to $(-\mu$ and $\mu)$ for every $k$ and for every $i=1,\dots,n.$ To prove that if $\alpha>2$ the method is in general not convergent it is enough to consider the case when $\theta_k = \theta_0$ for every iteration index $k$. In this case we have that $A_k = A_0$ for every $k$ and thus $\xi^{k}= A_0^k\xi^0$. In this case we can see \cite{dsg} that $1-\alpha$ is an eigenvalue of $A_0$ an therefore if $\alpha>2$ we have that  $\rho(A_0)>1$ and thus the sequence $\{\xi^k\}$ does not converge.
 This concludes the first part of the proof.\\

Assume now that the step-sizes are computed as in \eqref{spectralstep}. Proceeding as in the proof of Proposition 4.3 in \cite{dsg} we can prove that $\sigma^{k+1}_i = \sigma^{k+1}$ for every $i$ with $\sigma_{k+1}$ given by
$$\sigma_{k+1} = \begin{cases}
\min\{\sigma_{max}, 1+\sigma_{max}\theta_k\} & \text{if }\ \sigma_k = \sigma_{max}\\
\min\{\sigma_{max}, \hat\sigma^{k+1}\} & \text{otherwise}
\end{cases} $$
where $\hat\sigma^{k+1} = 1+\theta_k+\theta_k\theta{k-1}+\dots+\prod_{j=1}^k\theta^j+\sigma^0\prod_{j=0}^k\theta^j.$
By using the fact that $\theta_k > 1/3$ and $\sigma_{max} = 3/2$ we can prove that there exists $\bar k$ such that $\sigma^k = \sigma_{max}$ for every $k>\bar k.$ Therefore, for $k>\bar k$ the step-size becomes the same for all nodes and equal to $d^k_i = \sigma_{max}^{-1} = 2/3$ and thus the method converges by the first part of the Lemma. \end{proof}
\end{lemma}

The above Lemma certifies convergence 
of the spectral-like method \cite{dsg} for 
time-varying networks and a very specific problem structure with 
all-to-all communication network and consensus quadratic costs.
 It is worth noting that, for generic quadratic cost functions and sparse time-varying networks, an upper bound on the step-size is necessary (see Figures 1 and 2 below). 
We now make an analogy on the achieved results for the distributed spectral-like method \cite{dsg} and the spectral (Barzilei-Borwein) 
 gradient method from the centralized optimization.
Namely,  
in centralized settings, the spectral gradient method's  convergence without step size safeguarding has been proved only for a \emph{strictly convex quadratic cost function}.   In the case of generic functions beyond strictly convex quadratic, 
some safeguards $\Delta_{min}$ and $\Delta_{max}$ are necessary, even in the centralized case.  
Though, in the centralized case, these safeguards can be arbitrarily small ($\Delta_{min}$) and 
arbitrarily large ($\Delta_{max}$). 
Therefore, the need for safeguards is to be expected in the distributed optimization scenario as well. This matches with the results that we present here. 
 It turns our that the price to be payed in the distributed time-varying networks scenario is two-fold: 1) 
 the no-safeguards case happens in a more restricted cost functions setting, namely the consensus quadratic costs (see Lemma 4); and 2) the safeguard step size bounds in the general case are no longer arbitrary 
 and take a network-dependent form.\\
 

We also have the following Lemma where we continue to assume the consensus problem but relax the requirement that the network is fully connected at all times. 
 When the network is not fully connected, in general we need safeguarding for global convergence. However, as explained below, the following Lemma sheds some light on the behavior of the spectral-like distributed method. While it is not to be considered as a global convergence result, it highlights that the next step size has a controlled length provided that the current solution estimate is close to consensus.

\begin{lemma}
Let us assume that the objective function is given by \eqref{simpleproblem}, and that $x^0,z^0$ are such that $e^Tx^0 = e^Ta$, $z^0_i = \nabla f(x^0_i) = x^0_i$. Moreover, for every $i=1,\dots,n$ let the local stepsize $d^k_i$ be defined as $d^0_i = d^0>0$ and, for every $k\geq1$, $d^k_i = 1/\sigma^k_i$, with
\begin{equation}\label{sigma2}
    \sigma_i^{k+1} = 1+\sigma^k_i\sum_{j=1}^nw^k_{ij}\left(1-\frac{s^k_j}{s^k_i}\right)
\end{equation}
where $s^k_j = x^{k+1}_j-x^k_j$. Moreover, let us assume that at each iteration assumption A2 holds with $m=1$.\\
Given any $\hat d>1$, if $\|x^0-e\bar x^0\|\leq\hat\varepsilon$ with $\hat\varepsilon$ satisfying 
$$\hat\varepsilon\leq\frac{1}{\nu_0+d^0}\frac{(d^0)^2(\hat d-1)|\bar x^0|}{2\hat d + d^0(\hat d-1)}$$
then $d^1_i\leq\hat d$ for every $i=1,\dots,n.$
\begin{proof}
Let us denote with $J\in\R^{n\times n}$ the matrix $\frac{1}{n}ee^t$ and with $v^k = s^k-e\bar s^k$. From the assumptions and the double stochasticity of the matrix $W^k$, we have  
\begin{equation*}
    \begin{aligned}
    &\|v^0\| = \|(I-J)(x^1-x^0)\| = \|(I-W)(W^0x^0 - x^0-d^0z^0)\|\\
    &= \|(I-J)(W^0-I-d^0I)x^0\|\leq\|(W^0-I)x^0\|+d^0\|(I-J)x^0\|\\
    &\leq (\nu_0+d^0)\|x^0-e\bar x^0\|\leq (\nu_0+d^0)\hat\varepsilon = \varepsilon
    \end{aligned}
\end{equation*}
Where we defined $\varepsilon = (\nu_0+d^0)\hat\varepsilon. $
Moreover, $$ |\bar s^0| = \frac{1}{n}|e^t(W^0x^0-d^0x^0-x^0)| = d^0|\bar x^0|.$$
These imply that, for every $j=1,\dots,n$
$$1-\frac{2\varepsilon}{d^0|\bar x^0|-\varepsilon}\leq \frac{s^0_j}{s^0_i}\leq 1+\frac{2\varepsilon}{d^0|\bar x^0|-\varepsilon}.$$
Replacing these bounds in \eqref{sigma2}, and defining $\sigma^0 = 1/d^0$, we get
\begin{equation}\label{sigma_bounds}1-\frac{2\varepsilon\sigma^0}{d^0|\bar x^0|-\varepsilon}\leq \sigma^1_i\leq 1+\frac{2\varepsilon\sigma^0}{d^0|\bar x^0|-\varepsilon}.\end{equation}
It's easy to see that the first inequality, together with the assumption over $\hat \varepsilon$, imply $\sigma^1_i\geq 1/\hat d$, which in turn implies the thesis. 

\end{proof}
\end{lemma}

Intuitively, the Lemma above says that, for the considered problem, if algorithm \eqref{it} with stepsize \eqref{sigma2} starts from a
point close to consensus (i.e., a point where solution estimates across different nodes are mutually close), then 
the next step size at each node will not be too large. More precisely, the size of the next step size is controlled by the consensus neighborhood $\hat{\epsilon}$
 that we start from. In other words, if the next step size is to be upper bounded by an arbitrary constant $\hat{d}>1$, we can find a problem-dependent 
 constant $\hat{\epsilon}$ such that, starting at most 
 $\hat{\epsilon}$ away from consensus, the next step size at each node is 
 at most $\hat{d}$. To further explain this, 
suppose that all the quantities $s_j^k/s_i^k$'s are $\epsilon^\prime$-close to one, 
$|s_j^k/s_i^k| \in (1-\epsilon^\prime, 1+\epsilon^\prime)$, for all nodes $i,j$. Then, in view of \eqref{sigma_bounds}, quantity $\sigma_i^{k+1}$, for all nodes $i$, 
is approximated as:
\[
1 \pm  \sigma_i^k \,n\, \epsilon^\prime.
\]
In other words, for the special case of the consensus problem, 
 provided that all the quantities $s_j^k/s_i^k$'s are $\epsilon^\prime$-close to one,
  the next step-size $1/\sigma_i^{k+1}$ is in a neighborhood of 
  one, and is hence bounded.

We now present some numerical results. We consider the problem of minimizing a logistic loss function with $l_2$ regularization, that is, we assume the local objective function $f_{i}$ at node $i$ is given by 
\begin{equation}\label{logreg}
f_i(y) = \ln\left(1+\exp(-b_ia_i^Ty)\right) + \frac12R\|y\|^2_2
\end{equation}
where $a_i\in\R^d,\ b_i\in\{-1,1\}$ and $R>0.$
We compare 3 different choices of the matrix $B$ in \eqref{it} and three different definitions of the step-sizes $d^k_i$, resulting in nine methods. For increasing values of $\dM$ we run each method on the given problem and we plot in Figure 1 the number of iterations necessary to arrive at convergence. \\
The problem is generated as follows.
The convergence analysis we carried out in Section 3 does not rely on any particular definition of the step-sizes $d^k_i,$ therefore we need to specify how each node chooses the step-size at each iteration. We consider here two cases. The first one, referred to as \emph{spectral} in Figure 1, is the case where $d^k_i = (\sigma^k_i)^{-1}$ with $\sigma^k_i$ as in \eqref{spectralstep}. The second case we consider is the one where each node performs local line search by employing a backtracking strategy starting at $\dM$ to satisfy classical Armijo condition on the local objective function. That is, to satisfy 
$$ f_i\left(\sum_{j=1}^n w^k_{ij}x^k_j-d^k_i z^k_i \right)\leq f_i(x^k_i) - cd^k_i\nabla f_i(x^k_i)^Tz^k_i$$
with $c = 10^{-3}$ and $z^k_i = u^k_i+\nabla f_i(x^k_i).$
 We refer to this method as \emph{line search}.

 It is worth noting that there are no convergence guarantees for the 
 line search method. 
The rationale for including a comparison with it is to show that the method \cite{dsg} exhibits a significantly higher degree of robustness with respect to a meaningful, time-varying and node-varying, local step size strategy that can be employed.
For comparison, we also consider the method with fixed step-size $d^k_i = \dM $ for every $k$ and every $i=1,\dots,n.$ The choices of the matrix $B^k$ are given by $B^k=0$ (plot (a) in Figure 1), $\ B^k=\dM^{-1}I$ (plot (b)) and $B^k = \dM^{-1}\mathcal W^k$ (plot(c)), where for the case $B\neq 0$ the choice is made following \cite{dusan}. Notice that the case $d_i^k=d_{max}$ and $B^k=0$ corresponds to \cite{diging, harnessing} with constant, coordinated step-sizes. We consider increasing values of $\dM$ in $[\frac{1}{50L}, \frac{10}{L}]$, while we fix $\dm=10^{-8}$ as, in the considered framework, we saw that its choice does not influence the performance of the methods significantly. \\
In Figure 1 we plot the results in the case where the underlying network is symmetric and timevarying, defined as follows: we consider a network $G$ with $n=25$ nodes undirected and connected, generated as a random geometric graph with communication radius $\sqrt{n^{-1}\ln(n)}$, and we define the sequence of networks $\{G^k\}$ by deleting each edge with probability $\frac14.$ We carried out analogous tests in the cases where $G$ is symmetric and constant and in the case where it is given by a directed ring. The obtained results were comparable to the ones that we present.  We also observed in practice that double stochasticity of the underlying network appears to be essential for the convergence of the considered methods.\\
We set the dimension $d$ as equal to 10 and we generate the quantities involved in the definition of the local objective functions \eqref{logreg} as follows.
For $i = 1,\dots,n$ we define $a_i = (a_{i1},\dots,a_{i,d-1},1)^T$ where the components $a_{ij}$ are independent and come from the standard normal distribution, and $b_i = sign(a_i^Ty^*+\epsilon_i)$ where $y^*\in\R^d$ with independent components drawn from the standard normal distribution, and $\epsilon_i$ are generated according to the normal distribution with mean 0 and standard deviation 0.4. Finally, we take the regularization parameter $R = 0.25.$ The initial vectors $x^0_i$ are generated independently, with components drawn from the uniform distribution on $[0,1]$, and at each iteration we define the consensus matrix $W_k$ as the Metropolis matrix \cite{metropolis}.

We are interested in the number of iterations required by each method to reach a prescribed accuracy. More precisely, we evaluate the iteration number $\bar k$ at which 
$\max_{i=1,...,n}\|x_i^{\bar k}-y^*\|< \varepsilon$, 
where $\varepsilon =10^{-5}$.

In Figure 1, on the $x$-axis we show the upper bound $\dM$ while on the $y$-axis we show $\bar k$ for each method. To facilitate the comparison among the methods, in Figure 2 we plot the same results, with $y$-axis cut at 2000. \\

\begin{figure}[h]
    \begin{subfigure}[b]{0.48\textwidth}
      
        \includegraphics[width = \textwidth]{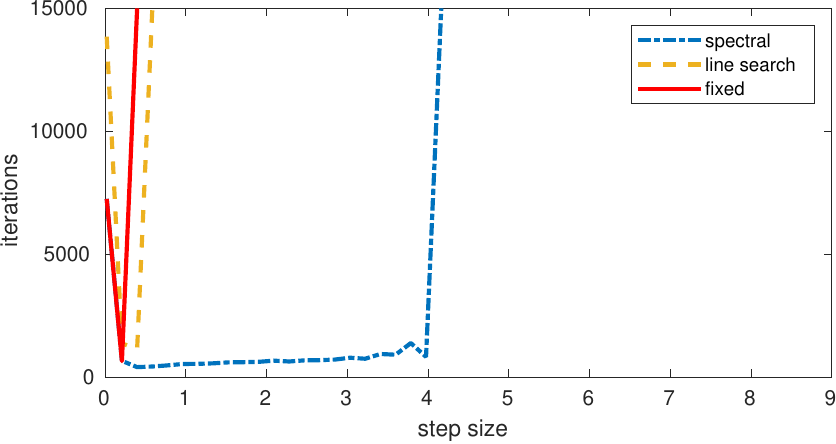}
        \caption{$B=0$}
    \end{subfigure}
    \hfill
    \begin{subfigure}[b]{0.48\textwidth}
        \includegraphics[width = \textwidth]{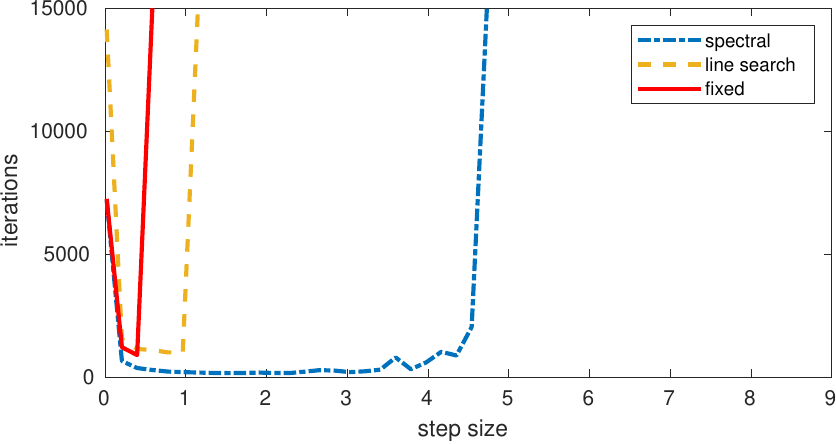}

        \caption{$B=bI$}
    \end{subfigure}
    \vskip\baselineskip
    \begin{subfigure}[b]{\textwidth}
        \centering
        \includegraphics[width = 0.48\textwidth]{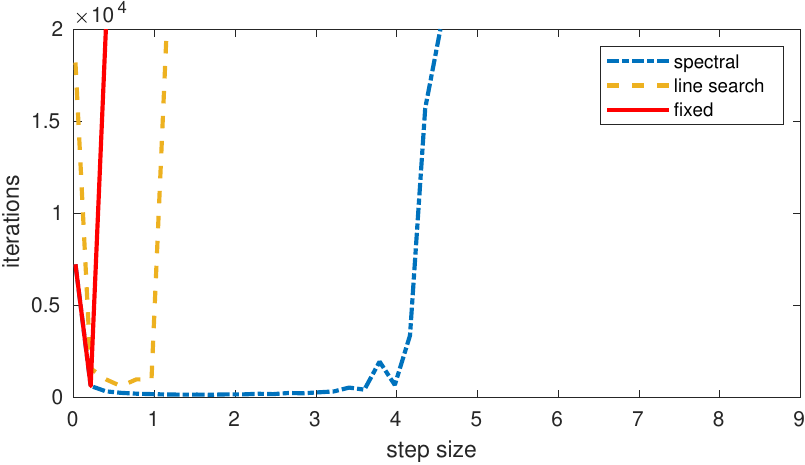}
		\caption{$B=bW$}
    \end{subfigure}
    \caption{}
\end{figure}

\begin{figure}[h]
    \begin{subfigure}[b]{0.48\textwidth}
      
        \includegraphics[width = \textwidth]{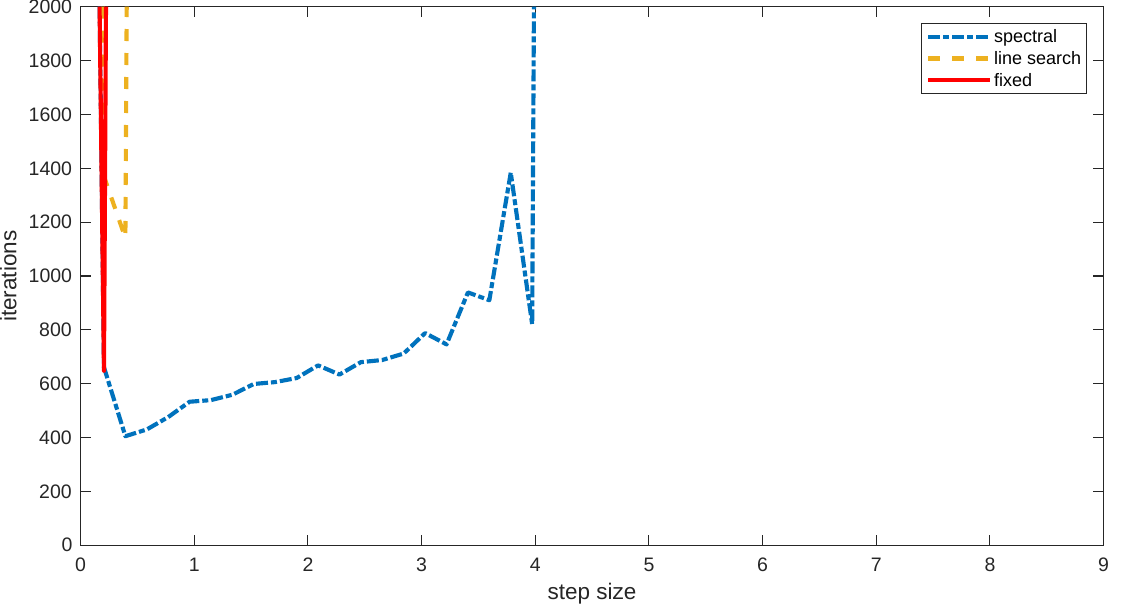}
        \caption{$B=0$}
    \end{subfigure}
    \hfill
    \begin{subfigure}[b]{0.48\textwidth}
        \includegraphics[width = \textwidth]{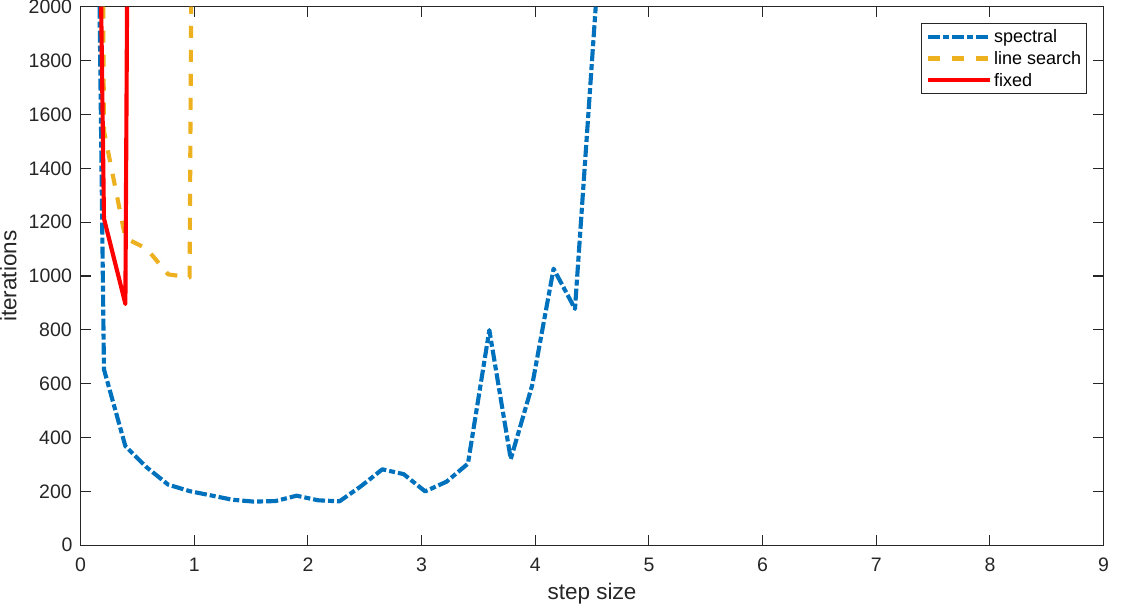}

        \caption{$B=bI$}
    \end{subfigure}
    \vskip\baselineskip
    \begin{subfigure}[b]{\textwidth}
        \centering
        \includegraphics[width = 0.48\textwidth]{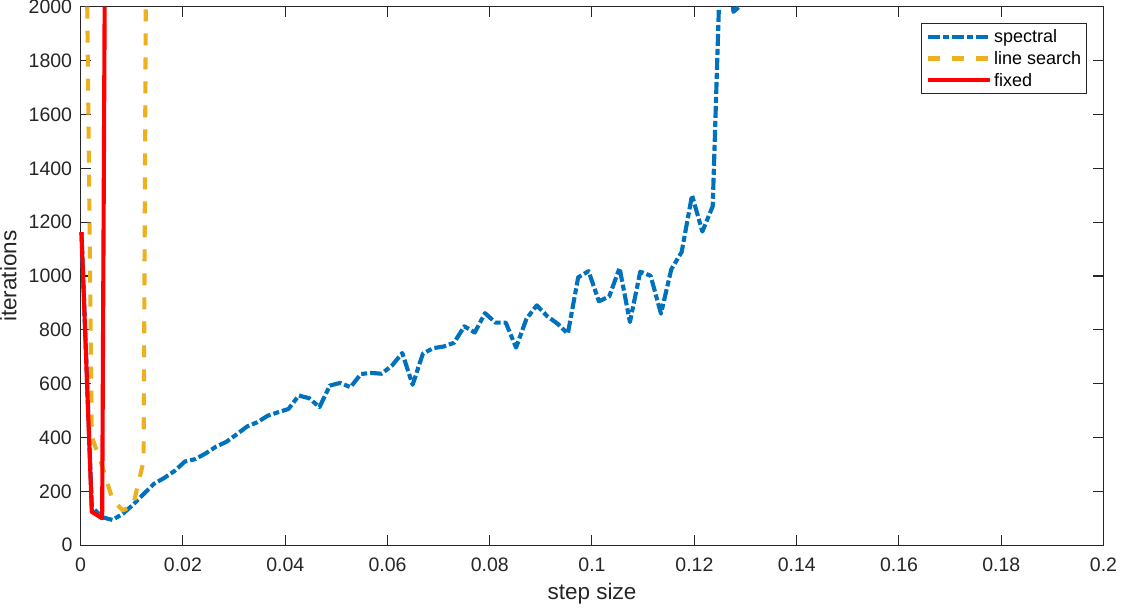}
		\caption{$B=bW$}
    \end{subfigure}
    \caption{}
\end{figure}

We can see from Figure 1 that for all considered choices of the matrix $B$ the spectral method allows for maximum step-size that is at least 10 times larger than the method with fixed step-size, while line search allows for maximum step-size equal to 2 and 3 times the maximum step-size allowed by the method with fixed steplength, for $B$ equal to $0$ and $bI$ or $b\mathcal{W}$ respectively. Moreover, we can see that choosing $B = bI$ seems to increase the maximum value of $\dM$ that yields convergence for all the considered methods. Finally, in Figure 2, we can notice that for most of the tested values of $\dM$ the spectral methods requires a smaller number of iterations than the method with fixed step-size. That is, in the considered framework, using uncoordinated time-varying step-sizes given by \cite{dsg} helps to significantly improve the robustness of the method and also the performance. Notice also that the spectral step-size strategy exhibits a ``stable", practically unchanged, performance for a wide range of $\dM$; hence, it is not sensitive to tuning of $\dM$. This is in contrast with the constant step-size strategy that is very sensitive to the step-size choice $\dM$. 
It is also worth noting that Theorem 2 requires a conservative upper bound on the step-size $\dM$ and a conservative upper bound on step-size differences $\Delta$ and that both depend on multiple global system parameters (Lemma 3). However, simulations presented here and other extensive numerical studies suggest that an a priori upper bound on $\Delta$ is not required for convergence. In addition, $\dm$ can be set to a small value independent of system parameters, e.g., $\dm =10^{-8}$, and setting $\dM$ requires only a coarse upper bound on quantity $1/L$.

\section{Conclusions}
We proved that a class of distributed first-order methods, including those proposed in \cite{dusan, dsg}, is robust to time-varying and uncoordinated step-sizes and time-varying weight-balanced digraphs, wherein connectedness of the network at each iteration, unlike, e.g., the recent work \cite{lessard}, is not required. The achieved results provide a solid improvement in understanding of the robustness of exact distributed first-order methods to time-varying networks and uncoordinated time-varying step-sizes. Most notably, we show that the unification strategy in \cite{dusan} and the spectral-like step-size selection strategy in \cite{dsg}, as well as combination of those, exhibits a high degree of robustness. This paper considers weight-balanced directed networks. Extensions to weight-imbalanced networks requires redefining the algorithmic class and the respective analysis, and represents an interesting future research direction.

\end{document}